\documentclass{amsart}

\usepackage{amssymb,latexsym}
\usepackage{amsfonts}
\usepackage{dsfont}
\usepackage{hyperref}
\usepackage{color}
\usepackage{graphicx}
\usepackage[all]{xy}
\usepackage{ytableau}

\usepackage{caption}
\usepackage{subcaption}

\usepackage{tikz}
\usetikzlibrary{calc, shapes, backgrounds,arrows,positioning,plotmarks}
\tikzset{>=stealth',
  head/.style = {fill = white, text=black},
  plaque/.style = {draw, rectangle, minimum size = 10mm, fill=white}, 
     pil/.style={->,thick},
  junct/.style = {draw,circle,inner sep=0.5pt,outer sep=0pt, fill=black}
  }

  \newcommand{\convexpath}[2]{
  [   
  create hullcoords/.code={
    \global\edef\namelist{#1}
    \foreach [count=\counter] \nodename in \namelist {
      \global\edef\numberofnodes{\counter}
      \coordinate (hullcoord\counter) at (\nodename);
    }
    \coordinate (hullcoord0) at (hullcoord\numberofnodes);
    \pgfmathtruncatemacro\lastnumber{\numberofnodes+1}
    \coordinate (hullcoord\lastnumber) at (hullcoord1);
  },
  create hullcoords
  ]
  ($(hullcoord1)!#2!-90:(hullcoord0)$)
  \foreach [
  evaluate=\currentnode as \previousnode using \currentnode-1,
  evaluate=\currentnode as \nextnode using \currentnode+1
  ] \currentnode in {1,...,\numberofnodes} {
    let \p1 = ($(hullcoord\currentnode) - (hullcoord\previousnode)$),
    \n1 = {atan2(\y1,\x1) + 90},
    \p2 = ($(hullcoord\nextnode) - (hullcoord\currentnode)$),
    \n2 = {atan2(\y2,\x2) + 90},
    \n{delta} = {Mod(\n2-\n1,360) - 360}
    in 
    {arc [start angle=\n1, delta angle=\n{delta}, radius=#2]}
    -- ($(hullcoord\nextnode)!#2!-90:(hullcoord\currentnode)$) 
  }
}

\theoremstyle{plain}
\newtheorem{theorem}{Theorem}[section]
\newtheorem{lemma}[theorem]{Lemma}

\newtheorem{proposition}[theorem]{Proposition}

\newtheorem{corollary}[theorem]{Corollary}
\theoremstyle{definition}
\newtheorem{remark}[theorem]{Remark}

\numberwithin{equation}{section}

\definecolor{darkblue}{rgb}{0.0,0,0.7}

\newcommand{\newword}[1]{\textcolor{darkblue}{\textbf{\emph{#1}}}}

\newcommand{\bfa}{\mathbf{a}}
\newcommand{\bfb}{\mathbf{b}}
\newcommand{\bfc}{\mathbf{c}}

\newcommand{\Gr}{\ensuremath{\mathrm{Gr}}}

\newcommand{\inc}{\ensuremath{\mathrm{Inc}}}
\newcommand{\incgl}{\ensuremath{\mathrm{Inc}_\mathrm{pack}}}

\newcommand{\PD}{\ensuremath{\mathcal{PD}}}
\newcommand{\BT}{\ensuremath{\mathcal{BT}}}
\newcommand{\TT}{\ensuremath{\mathcal{TT}}}
\newcommand{\DT}{\ensuremath{\mathcal{DT}}}
\newcommand{\rank}{\ensuremath{\mathrm{rank}}}

\begin{document}

\title[Dynamics of minuscule plane partitions]{Minuscule analogues of the plane partition periodicity conjecture of Cameron and Fon-Der-Flaass}

\author{Oliver Pechenik}
\address[OP]{Department of Combinatorics \& Optimization, University of Waterloo, Waterloo, ON N2L 3G1, Canada}
\email{oliver.pechenik@uwaterloo.ca}

\date{\today}

\begin{abstract}
Let $P$ be a graded poset of rank $r$ and let $\bfc$ be a $c$-element chain. For an order ideal $I$ of $P \times \bfc$, its rowmotion $\psi(I)$ is the smallest ideal containing the minimal elements of the complementary filter of $I$. The map $\psi$ defines invertible dynamics on the set of ideals.
We say that $P$ has \emph{NRP (`not relatively prime') rowmotion}  if no $\psi$-orbit has cardinality relatively prime to $r+c+1$. 

In work with R.~Patrias (2020), we proved a 1995 conjecture of P.~Cameron and D.~Fon-Der-Flaass by establishing NRP rowmotion for the product $P = \bfa \times \bfb$ of two chains, the poset whose order ideals correspond to the Schubert varieties of a Grassmann variety $\Gr_a(\mathbb{C}^{a+b})$ under containment. Here, we initiate the general study of posets with NRP rowmotion. 

Our first main result establishes NRP rowmotion for all minuscule posets $P$, posets whose order ideals reflect the Schubert stratification of minuscule flag varieties. Our second main result is that NRP promotion depends only on the isomorphism class of the comparability graph of $P$.
 \end{abstract}

\maketitle

\section{Introduction}\label{sec:intro}

The \emph{rowmotion} operator $\psi$ defines a discrete dynamical system on the (order) ideals of a finite poset. This operator first appeared in the 1970s \cite{Duchet, Brouwer.Schrijver}, but has gained more attention recently (e.g, \cite{Panyushev, Armstrong.Stump.Thomas, Striker.Williams,  Grinberg.Roby:II, Grinberg.Roby:I, Propp.Roby, Dilks.Pechenik.Striker, Thomas.Williams, Hopkins:symmetry}), with somewhat mysterious connections to cluster algebras \cite{Galashin.Pylyavskyy, Shen.Weng},  Schubert calculus \cite{Buch.Wang}, and quiver representations \cite{Garver.Patrias.Thomas}.

Let $Q$ be a finite poset and $I \subseteq Q$ be an ideal. Then its \newword{rowmotion} $\psi(I)$ is the ideal generated by the minimal elements of $Q \setminus I$. As $\psi$ is an invertible operator, it permutes the set $J(Q)$ of ideals of $Q$. For a general poset $Q$, very little is known about the orbit structure of this permutation. Nonetheless, many posets with a major role in algebraic combinatorics exhibit rich structure in their $\psi$-dynamics.

In this paper, our focus is on the family of \emph{minuscule posets} and their products with chains. In particular, all posets we consider will be \emph{finite} and \emph{graded}.
Minuscule posets encode the containments of Schubert varieties in \emph{minuscule flag varieties}, analogues of Grassmannians that exhibit many of their particularly attractive geometric properties. We recall the definition of minuscule posets and their classification in Section~\ref{sec:minuscule}. The minuscule posets corresponding to actual Grassmannians are the \newword{rectangles} $\bfa \times \bfb$, where $\bfa, \bfb$ denote chains of any positive cardinalities $a,b$. Empirically, general minuscule posets share many of the combinatorially nice properties of rectangles, and results first established for rectangles often extend to the general minuscule setting (see, e.g., \cite{Thomas.Yong:comin, Buch.Samuel, Rush.Shi} for examples of extensions of this form).

The \newword{length} of a chain poset is its number of cover relations, so one less than its number of elements.
The \newword{rank} of the poset $Q$ is the length $\rank(Q)$ of the longest chain in $Q$. For example, the rectangle $\bfa \times \bfb$ has rank $a+b-2$. For every graded poset $Q$, it is easy to see that the $\psi$-orbit of the empty order ideal has size $\rank(Q)+2$. In particular, if $P$ is a graded poset and $\bfc$ is a $c$-element chain, then the empty order ideal of $P \times \bfc$ has a $\psi$-orbit of cardinality exactly $\rank(P)+c+1$. Other order ideals may have $\psi$-orbits of different sizes, but in important cases the orbit cardinalities are all related to the quantity $\rank(P)+c+1$.

Say that a finite graded poset $P$ has \newword{NRP (`not relatively prime') rowmotion} if, for all positive $c$, no $\psi$-orbit of $J(P \times \bfc)$ has cardinality relatively prime to $\rank(P)+c+1$. In work with R.~Patrias, we proved that all rectangles have NRP rowmotion.

\begin{theorem}[{\cite{Patrias.Pechenik}}]\label{thm:rectangle}
	If $P = \bfa \times \bfb$ is a rectangle, then $P$ has NRP rowmotion. That is, for any $c > 0$, there is no $\psi$-orbit of $J(P \times \bfc)$ with cardinality relatively prime to $a+b+c-1$.
\end{theorem}

The special case of Theorem~\ref{thm:rectangle} with $a+b+c-1$ a prime number was conjectured in 1995 by P.~Cameron and D.~Fon-Der-Flaass \cite{Cameron.Fonderflaass}. As a corollary, Theorem~\ref{thm:rectangle} yields the only known proof of their conjecture.

In light of Theorem~\ref{thm:rectangle}, it is natural to ask what other posets exhibit NRP rowmotion. The goal of this paper is begin addressing this classification question. Our first main result is an analogue of Theorem~\ref{thm:rectangle} for arbitrary minuscule posets:

\begin{theorem}\label{thm:main}
	Let $M$ be a minuscule poset. Then $M$ has NRP rowmotion.
\end{theorem}

It was not {\it a priori} clear to the author whether NRP rowmotion should be common or not. However, it appears that posets with NRP rowmotion are very rare. In Section~\ref{sec:other}, we exhibit counterexamples to various {\it a priori} plausible extensions of Theorem~\ref{thm:main}. On the other hand, we do find a few non-minuscule posets with NRP rowmotion. In particular, we show in Theorem~\ref{thm:iso_graphs} that NRP rowmotion depends only of the graph-isomorphism class of the comparability graph of the poset $P$. Classifying posets with NRP rowmotion would be valuable, but looks to be difficult. 

Unfortunately, our proof of Theorem~\ref{thm:main} is not type-uniform, but relies on the classification of minuscule posets, with corresponding case-by-case analysis. In the exceptional types, it relies on the author's explicit computer calculations reported in \cite{Mandel.Pechenik}. It would be very interesting to have a uniform proof of Theorem~\ref{thm:main}. Likely, giving a uniform proof of (at least the minuscule case of) \cite[Conjecture~44]{Ilango.Pechenik.Zlatin} would be a key step in such an argument; however, that conjecture remains mysterious.

{\bf This paper is organized as follows.} In Section~\ref{sec:minuscule}, we recall the necessary background on minuscule posets, and give their classification. In Section~\ref{sec:tableaux}, we discuss the notions of $K$-promotion and $K$-evacuation on increasing tableaux. In Section~\ref{sec:proof}, we apply the results of Section~\ref{sec:tableaux} for arbitrary graded posets to the case of minuscule posets, culminating in a proof of Theorem~\ref{thm:main}. In this section, we also establish a minuscule analogue of the main theorem of \cite{Pechenik:frames}.
Finally, in Section~\ref{sec:other}, we explore posets with NRP rowmotion outside of the minuscule realm. Our main positive result in this final section is Theorem~\ref{thm:iso_graphs}, showing that NRP rowmotion is a feature of the comparability graph of $P$. We use this result to build some posets with NRP rowmotion from minuscule posets, and find a few other such posets. However, we also exhibit counterexamples to some naive potential generalizations of Theorem~\ref{thm:main}. 

\section{Minuscule posets}
\label{sec:minuscule}

In this section, we recall the definition and classification of minuscule posets. For additional background, see, e.g., \cite{Proctor:minuscule,Stembridge:minuscule,Thomas.Yong:comin,Rush.Shi,Mandel.Pechenik,Hamaker.Patrias.Pechenik.Williams,Okada}.

Let $G$ be a complex connected reductive group. Choose Borel subgroups $B_+, B_- \subseteq G$ intersecting along a maximal algebraic torus $T = B_+ \cap B_-$. This choice gives rise to a partition of the root system $\Phi$ of $G$ into two sets of equal size: the positive roots $\Phi^+$ and the negative roots $\Phi^-$. 

The set $\Phi^+$ has a natural poset structure, where we say $\alpha < \beta$ if $\beta - \alpha \in \Phi^+$. The minimal elements of this poset are the simple roots $\Delta \subseteq \Phi^+$; the simple roots are a linear basis for the span of $\Phi^+$. Certain simple roots $\delta \in \Delta$ are called \newword{minuscule}. These are the ones such that  the coroot $\delta^\vee$ appears with multiplicity at most one in the simple coroot expansion of every $\alpha^\vee$ for $\alpha \in \Phi^+$.

Every positive root $\alpha$ can be written uniquely as a linear combination of simple roots. For each minuscule simple root $\delta$, we obtain a corresponding minuscule poset $M_\delta$ by restricting the poset $\Phi^+$ to the set of positive roots whose simple root expansion involves $\delta$ with nonzero coefficient. 

Minuscule posets fall into three infinite families with an additional two exceptional examples. The cases where $G$ is a general linear group give rise to the rectangles $\bfa \times \bfb$ as minuscule posets. The cases where $G$ is an even-dimensional orthogonal group give rise to two families of minuscule posets, depending on the choice of minuscule root: the \newword{shifted staircases} and the \newword{propellers}. The shifted staircase $S_n$ may be constructed explicitly as the set $\{ (x,y) : x\leq y \in [n] \}$ under entrywise comparison, i.e., $(x,y) \leq (x',y')$ if $x \leq x'$ and $y \leq y'$. The propellers are the posets $J^k( {\bf 2} \times {\bf 2})$ obtained from the rectangle ${\bf 2} \times {\bf 2}$ by applying the order ideal functor $J$ arbitrarily-many times. The case $G = E_6$ yields the \newword{Cayley--Moufang poset} $J^2({\bf 3} \times {\bf 2})$. The case $G=E_7$ yields the \newword{Freudenthal poset} $J^3({\bf 3} \times {\bf 2})$. All other complex connected reductive groups $G$ turn out to merely recover posets already listed, so this paragraph in fact gives a complete accounting of all minuscule posets.

\definecolor{ududff}{rgb}{0.30196078431372547,0.30196078431372547,1.}
\definecolor{xfqqff}{rgb}{0.4980392156862745,0.,1.}
\definecolor{ffqqqq}{rgb}{1.,0.,0.}
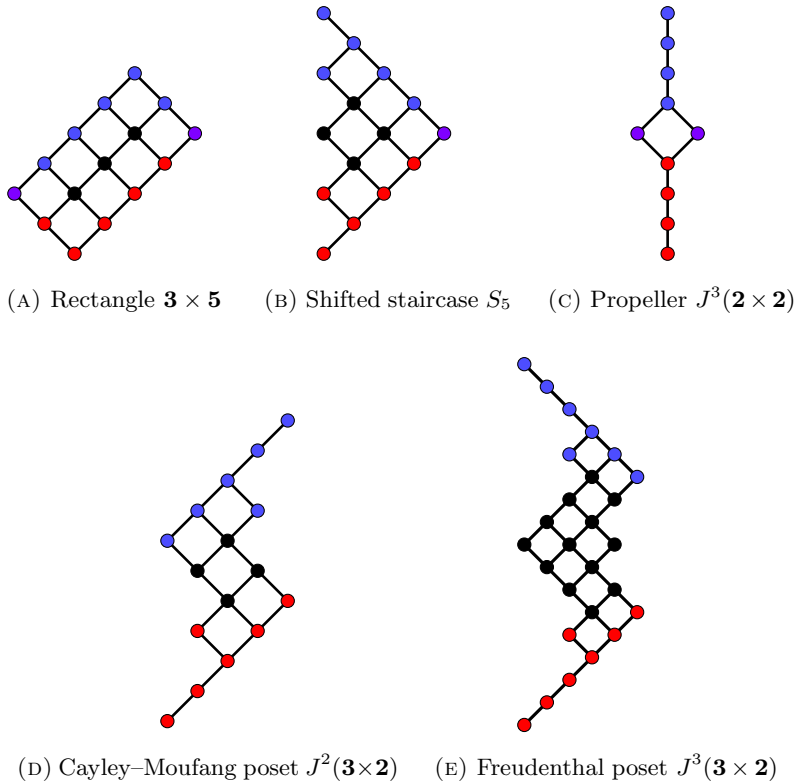
\begin{figure}[ht]
\centering
\begin{subfigure}[t]{1.2in}
	\begin{tikzpicture}[line cap=round,line join=round,>=triangle 45,x=2.0cm,y=2.0cm]
\clip(1.5,-.1) rectangle (3,1.7);
\draw [line width=1.pt] (2.,0.)-- (1.8,0.2);
\draw [line width=1.pt] (1.8,0.2)-- (1.6,0.4);
\draw [line width=1.pt] (2.,0.)-- (2.2,0.2);
\draw [line width=1.pt] (2.2,0.2)-- (2.,0.4);
\draw [line width=1.pt] (2.,0.4)-- (1.8,0.2);
\draw [line width=1.pt] (1.6,0.4)-- (1.8,0.6);
\draw [line width=1.pt] (1.8,0.6)-- (2.,0.4);
\draw [line width=1.pt] (2.,0.4)-- (2.2,0.6);
\draw [line width=1.pt] (1.8,0.6)-- (2.,0.8);
\draw [line width=1.pt] (2.,0.8)-- (2.2,0.6);
\draw [line width=1.pt] (2.2,0.2)-- (2.4,0.4);
\draw [line width=1.pt] (2.4,0.4)-- (2.2,0.6);
\draw [line width=1.pt] (2.,0.8)-- (2.2,1.);
\draw [line width=1.pt] (2.2,1.)-- (2.4,1.2);
\draw [line width=1.pt] (2.4,1.2)-- (2.6,1.);
\draw [line width=1.pt] (2.6,1.)-- (2.8,0.8);
\draw [line width=1.pt] (2.8,0.8)-- (2.6,0.6);
\draw [line width=1.pt] (2.6,0.6)-- (2.4,0.4);
\draw [line width=1.pt] (2.2,0.6)-- (2.4,0.8);
\draw [line width=1.pt] (2.2,1.)-- (2.4,0.8);
\draw [line width=1.pt] (2.4,0.8)-- (2.6,0.6);
\draw [line width=1.pt] (2.4,0.8)-- (2.6,1.);
\begin{scriptsize}
\draw [fill=ffqqqq] (2.,0.) circle (2.5pt);
\draw [fill=ffqqqq] (2.2,0.2) circle (2.5pt);
\draw [fill=ffqqqq] (1.8,0.2) circle (2.5pt);
\draw [fill=xfqqff] (1.6,0.4) circle (2.5pt);
\draw [fill=black] (2.,0.4) circle (2.5pt);
\draw [fill=ffqqqq] (2.4,0.4) circle (2.5pt);
\draw [fill=ffqqqq] (2.6,0.6) circle (2.5pt);
\draw [fill=xfqqff] (2.8,0.8) circle (2.5pt);
\draw [fill=ududff] (1.8,0.6) circle (2.5pt);
\draw [fill=ududff] (2.,0.8) circle (2.5pt);
\draw [fill=ududff] (2.2,1.) circle (2.5pt);
\draw [fill=ududff] (2.4,1.2) circle (2.5pt);
\draw [fill=ududff] (2.6,1.) circle (2.5pt);
\draw [fill=black] (2.4,0.8) circle (2.5pt);
\draw [fill=black] (2.2,0.6) circle (2.5pt);
\end{scriptsize}
\end{tikzpicture}
\caption{Rectangle ${\bf 3} \times {\bf 5}$}\label{fig:1a}
\end{subfigure}
\quad
\begin{subfigure}[t]{1.3in}
\begin{tikzpicture}[line cap=round,line join=round,>=triangle 45,x=2.0cm,y=2.0cm]
\clip(0.6,-0.1) rectangle (1.9,1.7);
\draw [line width=1.pt] (1.,0.)-- (1.2,0.2);
\draw [line width=1.pt] (1.,0.4)-- (1.2,0.2);
\draw [line width=1.pt] (1.2,0.2)-- (1.4,0.4);
\draw [line width=1.pt] (1.4,0.4)-- (1.6,0.6);
\draw [line width=1.pt] (1.6,0.6)-- (1.8,0.8);
\draw [line width=1.pt] (1.,0.4)-- (1.2,0.6);
\draw [line width=1.pt] (1.2,0.6)-- (1.4,0.4);
\draw [line width=1.pt] (1.,0.8)-- (1.2,0.6);
\draw [line width=1.pt] (1.2,0.6)-- (1.4,0.8);
\draw [line width=1.pt] (1.4,0.8)-- (1.6,0.6);
\draw [line width=1.pt] (1.4,0.8)-- (1.6,1.);
\draw [line width=1.pt] (1.6,1.)-- (1.8,0.8);
\draw [line width=1.pt] (1.4,0.8)-- (1.2,1.);
\draw [line width=1.pt] (1.2,1.)-- (1.4,1.2);
\draw [line width=1.pt] (1.4,1.2)-- (1.6,1.);
\draw [line width=1.pt] (1.,0.8)-- (1.2,1.);
\draw [line width=1.pt] (1.2,1.)-- (1.,1.2);
\draw [line width=1.pt] (1.,1.2)-- (1.2,1.4);
\draw [line width=1.pt] (1.,1.6)-- (1.2,1.4);
\draw [line width=1.pt] (1.2,1.4)-- (1.4,1.2);
\begin{scriptsize}
\draw [fill=ffqqqq] (1.,0.) circle (2.5pt);
\draw [fill=ffqqqq] (1.2,0.2) circle (2.5pt);
\draw [fill=ffqqqq] (1.4,0.4) circle (2.5pt);
\draw [fill=ffqqqq] (1.6,0.6) circle (2.5pt);
\draw [fill=xfqqff] (1.8,0.8) circle (2.5pt);
\draw [fill=ududff] (1.6,1.) circle (2.5pt);
\draw [fill=ududff] (1.4,1.2) circle (2.5pt);
\draw [fill=ududff] (1.2,1.4) circle (2.5pt);
\draw [fill=ududff] (1.,1.6) circle (2.5pt);
\draw [fill=ffqqqq] (1.,0.4) circle (2.5pt);
\draw [fill=black] (1.2,0.6) circle (2.5pt);
\draw [fill=black] (1.4,0.8) circle (2.5pt);
\draw [fill=black] (1.,0.8) circle (2.5pt);
\draw [fill=black] (1.2,1.) circle (2.5pt);
\draw [fill=ududff] (1.,1.2) circle (2.5pt);
\end{scriptsize}
\end{tikzpicture}
\caption{Shifted staircase $S_5$}\label{fig:1b}
\end{subfigure}
\quad
\begin{subfigure}[t]{1.3in}
\begin{tikzpicture}[line cap=round,line join=round,>=triangle 45,x=2.0cm,y=2.0cm]
\clip(0.2,-0.1) rectangle (1.4,1.7);
\draw [line width=1.pt] (1.,0.)-- (1.,0.2);
\draw [line width=1.pt] (1.,0.2)-- (1.,0.4);
\draw [line width=1.pt] (1.,0.4)-- (1.,0.6);
\draw [line width=1.pt] (1.,0.6)-- (0.8,0.8);
\draw [line width=1.pt] (0.8,0.8)-- (1.,1.);
\draw [line width=1.pt] (1.,1.)-- (1.2,0.8);
\draw [line width=1.pt] (1.2,0.8)-- (1.,0.6);
\draw [line width=1.pt] (1.,1.)-- (1.,1.2);
\draw [line width=1.pt] (1.,1.2)-- (1.,1.4);
\draw [line width=1.pt] (1.,1.6)-- (1.,1.4);
\begin{scriptsize}
\draw [fill=ffqqqq] (1.,0.) circle (2.5pt);
\draw [fill=ffqqqq] (1.,0.2) circle (2.5pt);
\draw [fill=ffqqqq] (1.,0.4) circle (2.5pt);
\draw [fill=ffqqqq] (1.,0.6) circle (2.5pt);
\draw [fill=xfqqff] (1.2,0.8) circle (2.5pt);
\draw [fill=xfqqff] (0.8,0.8) circle (2.5pt);
\draw [fill=ududff] (1.,1.) circle (2.5pt);
\draw [fill=ududff] (1.,1.2) circle (2.5pt);
\draw [fill=ududff] (1.,1.4) circle (2.5pt);
\draw [fill=ududff] (1.,1.6) circle (2.5pt);
\end{scriptsize}
\end{tikzpicture}
\caption{Propeller $J^3({\bf 2} \times {\bf 2})$}\label{fig:1c}
\end{subfigure}

\begin{subfigure}[t]{2.0in}
\begin{tikzpicture}[line cap=round,line join=round,>=triangle 45,x=2.0cm,y=2.0cm]
\clip(0.0,-0.1) rectangle (1.9,2.3);
\draw [line width=1.pt] (1.,0.)-- (1.2,0.2);
\draw [line width=1.pt] (1.2,0.2)-- (1.4,0.4);
\draw [line width=1.pt] (1.4,0.4)-- (1.2,0.6);
\draw [line width=1.pt] (1.2,0.6)-- (1.4,0.8);
\draw [line width=1.pt] (1.4,0.8)-- (1.6,0.6);
\draw [line width=1.pt] (1.6,0.6)-- (1.4,0.4);
\draw [line width=1.pt] (1.4,0.8)-- (1.6,1.);
\draw [line width=1.pt] (1.6,1.)-- (1.8,0.8);
\draw [line width=1.pt] (1.8,0.8)-- (1.6,0.6);
\draw [line width=1.pt] (1.6,1.)-- (1.4,1.2);
\draw [line width=1.pt] (1.4,1.2)-- (1.2,1.);
\draw [line width=1.pt] (1.2,1.)-- (1.4,0.8);
\draw [line width=1.pt] (1.,1.2)-- (1.2,1.);
\draw [line width=1.pt] (1.,1.2)-- (1.2,1.4);
\draw [line width=1.pt] (1.2,1.4)-- (1.4,1.2);
\draw [line width=1.pt] (1.4,1.2)-- (1.6,1.4);
\draw [line width=1.pt] (1.2,1.4)-- (1.4,1.6);
\draw [line width=1.pt] (1.4,1.6)-- (1.6,1.4);
\draw [line width=1.pt] (1.4,1.6)-- (1.6,1.8);
\draw [line width=1.pt] (1.6,1.8)-- (1.8,2.);
\begin{scriptsize}
\draw [fill=ffqqqq] (1.,0.) circle (2.5pt);
\draw [fill=ffqqqq] (1.2,0.2) circle (2.5pt);
\draw [fill=ffqqqq] (1.4,0.4) circle (2.5pt);
\draw [fill=ffqqqq] (1.6,0.6) circle (2.5pt);
\draw [fill=ffqqqq] (1.8,0.8) circle (2.5pt);
\draw [fill=black] (1.6,1.) circle (2.5pt);
\draw [fill=black] (1.4,1.2) circle (2.5pt);
\draw [fill=ududff] (1.2,1.4) circle (2.5pt);
\draw [fill=ududff] (1.4,1.6) circle (2.5pt);
\draw [fill=ududff] (1.6,1.8) circle (2.5pt);
\draw [fill=ududff] (1.8,2.) circle (2.5pt);
\draw [fill=ffqqqq] (1.2,0.6) circle (2.5pt);
\draw [fill=black] (1.4,0.8) circle (2.5pt);
\draw [fill=black] (1.2,1.) circle (2.5pt);
\draw [fill=ududff] (1.6,1.4) circle (2.5pt);
\draw [fill=ududff] (1.,1.2) circle (2.5pt);
\end{scriptsize}
\end{tikzpicture}
\caption{Cayley--Moufang poset $J^2({\bf 3} \times {\bf 2})$}\label{fig:1d}
\end{subfigure}
\quad
\begin{subfigure}[t]{1.8in}
\begin{tikzpicture}[line cap=round,line join=round,>=triangle 45,x=1.5cm,y=1.5cm]
\clip(0.2,-0.1) rectangle (2.3,3.6);
\draw [line width=1.2pt] (1.,3.2)-- (1.2,3.);
\draw [line width=1.2pt] (1.2,3.)-- (1.4,2.8);
\draw [line width=1.2pt] (1.4,2.8)-- (1.6,2.6);
\draw [line width=1.2pt] (1.6,2.6)-- (1.4,2.4);
\draw [line width=1.2pt] (1.4,2.4)-- (1.6,2.2);
\draw [line width=1.2pt] (1.6,2.2)-- (1.8,2.4);
\draw [line width=1.2pt] (1.6,2.6)-- (1.8,2.4);
\draw [line width=1.2pt] (1.8,2.4)-- (2.,2.2);
\draw [line width=1.2pt] (2.,2.2)-- (1.8,2.);
\draw [line width=1.2pt] (1.6,2.2)-- (1.4,2.);
\draw [line width=1.2pt] (1.4,2.)-- (1.6,1.8);
\draw [line width=1.2pt] (1.8,2.)-- (1.6,1.8);
\draw [line width=1.2pt] (1.6,2.2)-- (1.8,2.);
\draw [line width=1.2pt] (1.4,2.)-- (1.2,1.8);
\draw [line width=1.2pt] (1.2,1.8)-- (1.,1.6);
\draw [line width=1.2pt] (1.,1.6)-- (1.2,1.4);
\draw [line width=1.2pt] (1.2,1.4)-- (1.4,1.6);
\draw [line width=1.2pt] (1.4,1.6)-- (1.2,1.8);
\draw [line width=1.2pt] (1.4,1.6)-- (1.6,1.8);
\draw [line width=1.2pt] (1.6,1.8)-- (1.8,1.6);
\draw [line width=1.2pt] (1.8,1.6)-- (1.6,1.4);
\draw [line width=1.2pt] (1.4,1.6)-- (1.6,1.4);
\draw [line width=1.2pt] (1.6,1.4)-- (1.8,1.2);
\draw [line width=1.2pt] (1.2,1.4)-- (1.4,1.2);
\draw [line width=1.2pt] (1.4,1.2)-- (1.6,1.4);
\draw [line width=1.2pt] (1.,0.)-- (1.2,0.2);
\draw [line width=1.2pt] (1.2,0.2)-- (1.4,0.4);
\draw [line width=1.2pt] (1.4,0.4)-- (1.6,0.6);
\draw [line width=1.2pt] (1.6,0.6)-- (1.8,0.8);
\draw [line width=1.2pt] (1.8,0.8)-- (2.,1.);
\draw [line width=1.2pt] (1.8,1.2)-- (1.6,1.);
\draw [line width=1.2pt] (1.6,1.)-- (1.4,1.2);
\draw [line width=1.2pt] (1.4,0.8)-- (1.6,0.6);
\draw [line width=1.2pt] (1.4,0.8)-- (1.6,1.);
\draw [line width=1.2pt] (1.6,1.)-- (1.8,0.8);
\draw [line width=1.2pt] (1.8,1.2)-- (2.,1.);
\begin{scriptsize}
\draw [fill=ffqqqq] (1.,0.) circle (2.5pt);
\draw [fill=ffqqqq] (1.2,0.2) circle (2.5pt);
\draw [fill=ffqqqq] (1.4,0.4) circle (2.5pt);
\draw [fill=ffqqqq] (1.6,0.6) circle (2.5pt);
\draw [fill=ffqqqq] (1.8,0.8) circle (2.5pt);
\draw [fill=ffqqqq] (2.,1.) circle (2.5pt);
\draw [fill=black] (1.8,1.2) circle (2.5pt);
\draw [fill=black] (1.6,1.4) circle (2.5pt);
\draw [fill=black] (1.4,1.6) circle (2.5pt);
\draw [fill=black] (1.2,1.8) circle (2.5pt);
\draw [fill=ffqqqq] (1.4,0.8) circle (2.5pt);
\draw [fill=black] (1.6,1.) circle (2.5pt);
\draw [fill=black] (1.4,1.2) circle (2.5pt);
\draw [fill=black] (1.8,1.6) circle (2.5pt);
\draw [fill=black] (1.2,1.4) circle (2.5pt);
\draw [fill=black] (1.,1.6) circle (2.5pt);
\draw [fill=black] (1.4,2.) circle (2.5pt);
\draw [fill=black] (1.6,1.8) circle (2.5pt);
\draw [fill=black] (1.8,2.) circle (2.5pt);
\draw [fill=ududff] (2.,2.2) circle (2.5pt);
\draw [fill=black] (1.6,2.2) circle (2.5pt);
\draw [fill=ududff] (1.8,2.4) circle (2.5pt);
\draw [fill=ududff] (1.6,2.6) circle (2.5pt);
\draw [fill=ududff] (1.4,2.4) circle (2.5pt);
\draw [fill=ududff] (1.4,2.8) circle (2.5pt);
\draw [fill=ududff] (1.2,3.) circle (2.5pt);
\draw [fill=ududff] (1.,3.2) circle (2.5pt);
\end{scriptsize}
\end{tikzpicture}
\caption{Freudenthal poset $J^3({\bf 3} \times {\bf 2})$}\label{fig:1e}
\end{subfigure}

\caption{Example Hasse diagrams of the five families of minuscule posets. The bottom tree of each poset is given by the red and purple nodes, while each top tree is given by the blue and purple nodes. The black nodes of each poset are exactly the elements outside of its doubletree.}\label{fig:1}
\end{figure}

Corresponding to each simple root $\delta \in \Delta$ is a maximal parabolic group $P_\delta \supseteq B_+$. In the case $\delta$ is minuscule, we refer to the smooth projective variety $G / P_\delta$ as a \newword{minuscule variety}. For example, in the case that $G$ is a general linear group, the minuscule varieties obtained are Grassmannians. The $B_-$-orbits on $G / P_\delta$ are called \newword{Schubert cells} $\Omega_u$, indexed by elements of the parabolic Weyl group $W_{P_\delta}$. A \newword{Schubert variety} $X_u$ is the closure of the Schubert cell $\Omega_u$; each Schubert variety of $G / P_\delta$ is a union of Schubert cells. Hence, the Schubert varieties of $G / P_\delta$ form a poset under containment. In fact, this poset of Schubert varieties is isomorphic to $J(M_\delta)$, the distributive lattice of order ideals of $M_\delta$; equivalently, $M_\delta$ has an alternative construction as the subposet of join-irreducibles of this poset of subvarieties.

Each minuscule poset $M_\delta$ is self-dual. This duality may be seen explicitly by considering the classification of minuscule posets above.  Conceptually, it is induced by the action of the longest element $w_0^\delta$ of $W_{P_\delta}$ on $M_\delta$. We denote this anti-involution on $M_\delta$ by $\PD$ because it is closely related to the Poincar\'e duality on the compact manifold $G/ P_\delta$. For the rectangles and Cayley--Moufang poset as drawn in Figure~\ref{fig:1}, the anti-involution is given by $180^\circ$ rotation; for the remaining diagrams of Figure~\ref{fig:1}, it is given by reflection across a horizontal line.

The definitions of this paragraph are all essentially borrowed from \cite{Proctor:Dynkin}, where they appear in more generality. 
For $M_\delta$ a minuscule poset, say $x \in M_\delta$ is a \newword{bottom tree element} if the principal ideal generated by $x$ is a chain. Define the \newword{bottom tree} of $M_\delta$ to be the order ideal $\BT(M_\delta)$ of $M_\delta$ consisting of all bottom tree elements. Similarly, define the \newword{top tree} $\TT(M_\delta)$ to be the order filter consisting of all elements generating a principal order filter that is a chain. It is easy to see that $\TT(M_\delta)$ is the image of $\BT(M_\delta)$ under the action of the anti-involution $\PD$, and vice versa. Finally, define the \newword{doubletree} of $M_\delta$ to be the union $\DT(M_\delta) = \BT(M_\delta) \cup \TT(M_\delta)$. In the case that $M_\delta \cong \bfa \times \bfb$ is a rectangle, the doubletree was referred to in \cite{Pechenik:frames,Patrias.Pechenik} as the ``frame'' of the poset; however, that name seems to make less sense for other posets.

\section{Increasing tableaux}\label{sec:tableaux}

As in the proof of the main theorem of \cite{Patrias.Pechenik}, we will prove Theorem~\ref{thm:main} by first reformulating it in terms of $K$-promotion on increasing tableaux. In this section, we recall the relevant definitions and establish the general properties that we will need. In the following section, we will apply these results to the cases of minuscule posets.

Let $[q] = \{1, 2, \dots, q\}$.
An \newword{increasing tableau} (of height $q$) on the finite poset $P$ is a strictly order-preserving map $T : P \to [q]$, i.e.\ for $x<y \in P$, we have $T(x) < T(y)$. Under this name, increasing tableaux first appeared in \cite{Thomas.Yong:K}, although they have a longer history in other contexts (e.g., \cite{Stanley:ordered,Edelman.Greene}). We usually visualize an increasing tableau $T$ by labeling each poset element $x \in T$ with the value $T(x)$. We write $\inc^q(P)$ to denote the set of all increasing tableaux of height $q$ on the poset $P$. 

For $q <r$, we identify $\inc^q(P)$ with a subset of $\inc^r(P)$ in the obvious way. With these identifications, define $\inc(P) = \bigcup_q \inc^q(P)$. Say $T \in \inc(P)$ is \newword{packed} if it is a surjective map onto $[q]$ for some $q$. Note that the set $\incgl(P)$ of all packed increasing tableaux on $P$ is finite.

Thomas and Yong \cite{Thomas.Yong:K} developed a jeu de taquin theory for increasing tableaux, with application to the $K$-theoretic Schubert calculus of minuscule varieties \cite{Clifford.Thomas.Yong,Buch.Samuel}. Building on this theory, \cite{Pechenik:CSP} introduced the following definition of $K$-promotion. 

Let $T \in \inc^q(P)$. Consider the elements of $P$ labeled $1$ and $2$. This subset of the Hasse diagram of $P$ breaks up into connected components, called \newword{tiles}. Say a tile is \newword{trivial} if it has cardinality $1$, and \newword{nontrivial} otherwise. For each trivial tile, do nothing, while for each nontrivial tile, swap the labels $1$ and $2$. The result is an increasing tableau on $P$ with respect to the nonstandard order on $[q]$ where $2 < 1 <3 < \cdots$. Now, consider the connected components of the subposet with labels $1$ and $3$ and repeat this process, successively swapping the pairs of labels $(1,2), (1,3), (1,4), \dots, (1,q)$, obtaining an increasing tableau with respect to the nonstandard order $2< 3 < \cdots < q < 1$. Finally, decrement each label by $1$ and replace any label $0$ with the label $q$. The result is an increasing tableau of height $q$ on $P$ (with respect to the ordinary order on $[q]$), which we call the $K$-promotion $\psi(T)$ of $T$. Note that $\psi(T)$ depends on the choice of $q$ such that $T \in \inc^q(P)$. An example of the $K$-promotion process is shown in Figure~\ref{fig:Kpromotion}. For more details and examples of this $K$-promotion operator, see, e.g., \cite{Pechenik:CSP,Dilks.Pechenik.Striker,Mandel.Pechenik}. 

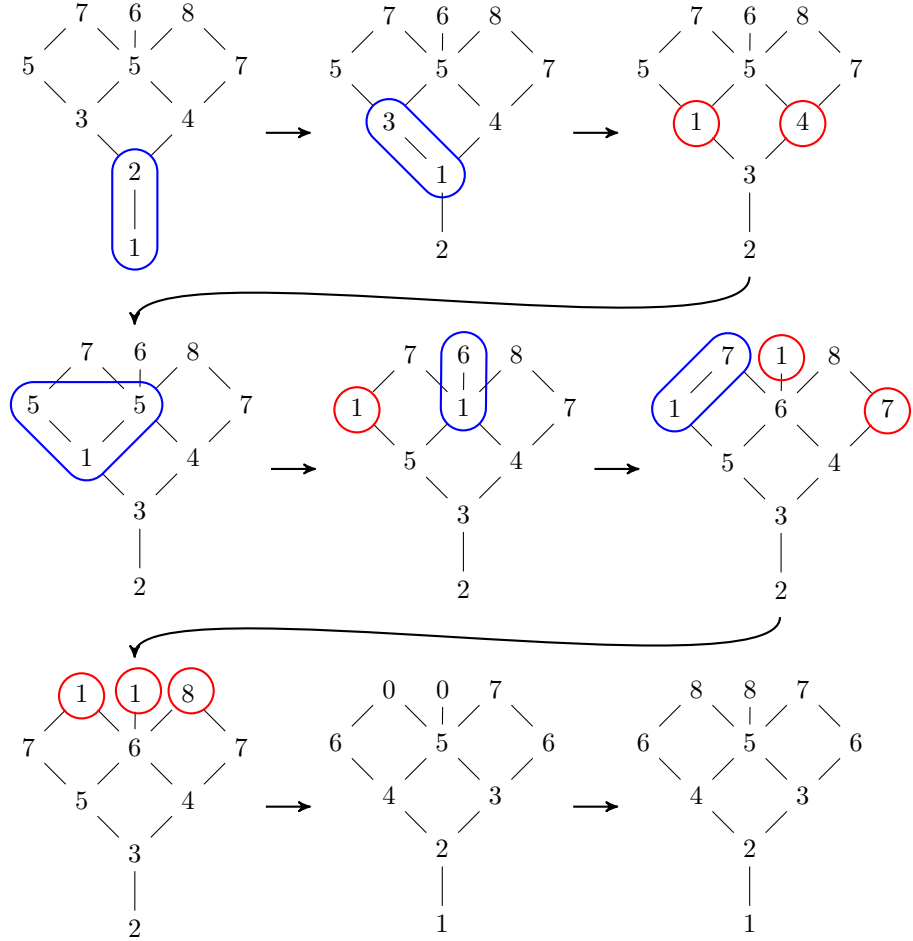
\begin{figure}[ht]
\begin{tikzpicture}
	\node (begin) at (0,0) {\begin{tikzpicture}
	\node (1) at (0,0) {$1$};
\node [above of=1] (2) {$2$};
\node [above left of=2] (3) {$3$};
\node [above right of=2] (5) {$4$};
\node [above left of=3] (4) {$5$};
\node [above right of=3] (6) {$5$};
\node [above right of=5] (8) {$7$};
\node [above right of=4] (7) {$7$};
\node [above right of=6] (9) {$8$};
\node [right= 0.29cm of 7] (10) {$6$};
\draw (1) to (2);
\draw (2) to (3);
\draw (2) to (5);
\draw (3) to (6);
\draw (5) to (6);
\draw (5) to (8);
\draw (8) to (9);
\draw (3) to (4);
\draw (4) to (7);
\draw (6) to (7);
\draw (6) to (9);
\draw (6) to (10);
\draw[thick,blue] \convexpath{1,2}{0.3cm};
\end{tikzpicture}
};
\node [right = 0.6cm of begin] (step1) {\begin{tikzpicture}
	\node (1) at (0,0) {$2$};
\node [above of=1] (2) {$1$};
\node [above left of=2] (3) {$3$};
\node [above right of=2] (5) {$4$};
\node [above left of=3] (4) {$5$};
\node [above right of=3] (6) {$5$};
\node [above right of=5] (8) {$7$};
\node [above right of=4] (7) {$7$};
\node [above right of=6] (9) {$8$};
\node [right= 0.29cm of 7] (10) {$6$};
\draw (1) to (2);
\draw (2) to (3);
\draw (2) to (5);
\draw (3) to (6);
\draw (5) to (6);
\draw (5) to (8);
\draw (8) to (9);
\draw (3) to (4);
\draw (4) to (7);
\draw (6) to (7);
\draw (6) to (9);
\draw (6) to (10);
\draw[thick,blue] \convexpath{2,3}{0.3cm};
\end{tikzpicture}
};
\node [right = 0.6cm of step1] (step2) {\begin{tikzpicture}
	\node (1) at (0,0) {$2$};
\node [above of=1] (2) {$3$};
\node [above left of=2] (3) {$1$};
\node [below left = -0.05cm of 3] (fake) {};
\node [above right of=2] (5) {$4$};
\node [below left = -0.05cm of 5] (fake2) {};
\node [above left of=3] (4) {$5$};
\node [above right of=3] (6) {$5$};
\node [above right of=5] (8) {$7$};
\node [above right of=4] (7) {$7$};
\node [above right of=6] (9) {$8$};
\node [right= 0.29cm of 7] (10) {$6$};
\draw (1) to (2);
\draw (2) to (3);
\draw (2) to (5);
\draw (3) to (6);
\draw (5) to (6);
\draw (5) to (8);
\draw (8) to (9);
\draw (3) to (4);
\draw (4) to (7);
\draw (6) to (7);
\draw (6) to (9);
\draw (6) to (10);
\draw[thick,red] \convexpath{fake}{0.3cm};
\draw[thick,red] \convexpath{fake2}{0.3cm};
\end{tikzpicture}
};
\node [below  = 0.6cm  of begin] (step3) {\begin{tikzpicture}
	\node (1) at (0,0) {$2$};
\node [above of=1] (2) {$3$};
\node [above left of=2] (3) {$1$};
\node [above right of=2] (5) {$4$};
\node [above left of=3] (4) {$5$};
\node [above right of=3] (6) {$5$};
\node [above right of=5] (8) {$7$};
\node [above right of=4] (7) {$7$};
\node [above right of=6] (9) {$8$};
\node [right= 0.29cm of 7] (10) {$6$};
\draw (1) to (2);
\draw (2) to (3);
\draw (2) to (5);
\draw (3) to (6);
\draw (5) to (6);
\draw (5) to (8);
\draw (8) to (9);
\draw (3) to (4);
\draw (4) to (7);
\draw (6) to (7);
\draw (6) to (9);
\draw (6) to (10);
\draw[thick,blue] \convexpath{3,4,6}{0.3cm};
\end{tikzpicture}
};
\node [right = 0.6cm of step3] (step4) {\begin{tikzpicture}
	\node (1) at (0,0) {$2$};
\node [above of=1] (2) {$3$};
\node [above left of=2] (3) {$5$};
\node [above right of=2] (5) {$4$};
\node [above left of=3] (4) {$1$};
\node [below left = -0.05cm of 4] (fake) {};
\node [above right of=3] (6) {$1$};
\node [above right of=5] (8) {$7$};
\node [above right of=4] (7) {$7$};
\node [above right of=6] (9) {$8$};
\node [right= 0.29cm of 7] (10) {$6$};
\draw (1) to (2);
\draw (2) to (3);
\draw (2) to (5);
\draw (3) to (6);
\draw (5) to (6);
\draw (5) to (8);
\draw (8) to (9);
\draw (3) to (4);
\draw (4) to (7);
\draw (6) to (7);
\draw (6) to (9);
\draw (6) to (10);
\draw[thick,blue] \convexpath{6,10}{0.3cm};
\draw[thick,red] \convexpath{fake}{0.3cm};
\end{tikzpicture}
};
\node [right = 0.6cm of step4] (step5) {\begin{tikzpicture}
	\node (1) at (0,0) {$2$};
\node [above of=1] (2) {$3$};
\node [above left of=2] (3) {$5$};
\node [above right of=2] (5) {$4$};
\node [above left of=3] (4) {$1$};
\node [below left = -0.05cm of 10] (fake) {};
\node [above right of=3] (6) {$6$};
\node [above right of=5] (8) {$7$};
\node [above right of=4] (7) {$7$};
\node [below left = -0.05cm of 8] (fake2) {};
\node [above right of=6] (9) {$8$};
\node [right= 0.29cm of 7] (10) {$1$};
\draw (1) to (2);
\draw (2) to (3);
\draw (2) to (5);
\draw (3) to (6);
\draw (5) to (6);
\draw (5) to (8);
\draw (8) to (9);
\draw (3) to (4);
\draw (4) to (7);
\draw (6) to (7);
\draw (6) to (9);
\draw (6) to (10);
\draw[thick,blue] \convexpath{4,7}{0.3cm};
\draw[thick,red] \convexpath{fake}{0.3cm};
\draw[thick,red] \convexpath{fake2}{0.3cm};
\end{tikzpicture}
};
\node [below = 0.6cm of step3] (step6) {\begin{tikzpicture}
	\node (1) at (0,0) {$2$};
\node [above of=1] (2) {$3$};
\node [above left of=2] (3) {$5$};
\node [above right of=2] (5) {$4$};
\node [above left of=3] (4) {$7$};
\node [below left = 0.19cm of 10] (fake) {};
\node [above right of=3] (6) {$6$};
\node [above right of=5] (8) {$7$};
\node [above right of=4] (7) {$1$};
\node [below left = -0.05cm of 7] (fake2) {};
\node [below left = 0.19cm of 9] (fake3) {};
\node [above right of=6] (9) {$8$};
\node [right= 0.29cm of 7] (10) {$1$};
\draw (1) to (2);
\draw (2) to (3);
\draw (2) to (5);
\draw (3) to (6);
\draw (5) to (6);
\draw (5) to (8);
\draw (8) to (9);
\draw (3) to (4);
\draw (4) to (7);
\draw (6) to (7);
\draw (6) to (9);
\draw (6) to (10);
\draw[thick,red] \convexpath{fake}{0.3cm};
\draw[thick,red] \convexpath{fake2}{0.3cm};
\draw[thick,red] \convexpath{fake3}{0.3cm};
\end{tikzpicture}
};
\node [right = 0.6cm of step6] (step7) {\begin{tikzpicture}
	\node (1) at (0,0) {$1$};
\node [above of=1] (2) {$2$};
\node [above left of=2] (3) {$4$};
\node [above right of=2] (5) {$3$};
\node [above left of=3] (4) {$6$};
\node [above right of=3] (6) {$5$};
\node [above right of=5] (8) {$6$};
\node [above right of=4] (7) {$0$};
\node [above right of=6] (9) {$7$};
\node [right= 0.29cm of 7] (10) {$0$};
\draw (1) to (2);
\draw (2) to (3);
\draw (2) to (5);
\draw (3) to (6);
\draw (5) to (6);
\draw (5) to (8);
\draw (8) to (9);
\draw (3) to (4);
\draw (4) to (7);
\draw (6) to (7);
\draw (6) to (9);
\draw (6) to (10);
\end{tikzpicture}
};
\node [right = 0.6cm of step7] (step8) {\begin{tikzpicture}
	\node (1) at (0,0) {$1$};
\node [above of=1] (2) {$2$};
\node [above left of=2] (3) {$4$};
\node [above right of=2] (5) {$3$};
\node [above left of=3] (4) {$6$};
\node [above right of=3] (6) {$5$};
\node [above right of=5] (8) {$6$};
\node [above right of=4] (7) {$8$};
\node [above right of=6] (9) {$7$};
\node [right= 0.29cm of 7] (10) {$8$};
\draw (1) to (2);
\draw (2) to (3);
\draw (2) to (5);
\draw (3) to (6);
\draw (5) to (6);
\draw (5) to (8);
\draw (8) to (9);
\draw (3) to (4);
\draw (4) to (7);
\draw (6) to (7);
\draw (6) to (9);
\draw (6) to (10);
\end{tikzpicture}
};
\draw[pil] (begin) to (step1);
\draw[pil] (step1) to (step2);
\draw[pil] (step2.south) .. controls ([yshift=-3cm] step2) and ([yshift=3cm] step3) .. (step3.north);
\draw[pil] (step3) to (step4);
\draw[pil] (step4) to (step5);
\draw[pil] (step5.south) .. controls ([yshift=-3cm] step5) and ([yshift=3cm] step6) .. (step6.north);
\draw[pil] (step6) to (step7);
\draw[pil] (step7) to (step8);
\end{tikzpicture}
\caption{An example of $K$-promotion on an increasing tableau $T \in \inc^8(H)$, where $H$ is the \emph{bee hummingbird} poset discussed in Section~\ref{sec:other}. At each step, the trivial tiles in circled in red, while the nontrivial tiles are circled in blue.}
\label{fig:Kpromotion}
\end{figure}

Our abuse of notation, using $\psi$ to denote both rowmotion of order ideals and $K$-promotion of increasing tableaux, is justified by the following theorem.

\begin{theorem}[\cite{Dilks.Pechenik.Striker, Mandel.Pechenik,Dilks.Striker.Vorland}]\label{thm:equivbij}
For any graded poset $P$ (such as a minuscule poset), there is a equivariant bijection between $J(P \times \bfc)$ under rowmotion and $\inc^{\rank(P) + c + 1}(P)$ under $K$-promotion.
\end{theorem}

\begin{remark}
	There is an easy bijection between  $J(P \times \bfc)$ and $\inc^{\rank(P) + c + 1}(P)$. If we identify $J(P \times \bfc)$ with weakly increasing labelings of $P$ and $\inc^{\rank(P) + c + 1}(P)$ with strictly increasing labelings of $P$, we easily biject weakly increasing labelings to strictly increasing ones by adding $r(x)$ to the label of each $x \in P$, where $r(x)$ denotes the number of elements in a largest chain of $P$ with maximum element $x$.
	
	This easy bijection, however, is \textbf{not} the equivariant bijection of Theorem~\ref{thm:equivbij}. Indeed, it is difficult to describe the equivariant bijection explicitly; the best tool for this is the \emph{multidimensional recombination} of \cite{Vorland}. For the purposes of this paper, however, it is happily sufficient to leave Theorem~\ref{thm:equivbij} in a non-constructive form.
\end{remark}

We will also need the following result describing how $K$-promotion of increasing tableaux is controlled by the cases of packed tableaux. For an increasing tableau $T \in \inc^q(P)$ using $d$ distinct labels, define the \newword{deflation} of $T$ to be the unique packed increasing tableau $T^\flat \in \inc^d(P)$ whose labels satisfy the same inequalities as $T$, i.e., we have, for $p,p' \in P$, that $T(p) < T(p')$ if and only if $T^\flat(p) < T^\flat(p')$. 

\begin{proposition}\label{prop:deflation}
	Let $T \in \inc^q(P)$ and suppose $T^\flat \in \inc^d(P)$. Then $\psi^q(T)$ uses the same set of $d$ distinct labels as $T$. Moreover, we have $\psi^q(T)^\flat=\psi^d(T^\flat)$.
\end{proposition}
\begin{proof}
	The fact that $T$ and $\psi^q(T)$ use the same set of $d$ distinct labels is proved explicitly for $P$ a rectangle as \cite[Lemma~2.1]{Dilks.Pechenik.Striker}. The proof given there extends without change to prove the corresponding fact for arbitrary finite posets.
	The last statement of the proposition is then immediate by applying \cite[Proposition~5.1]{Mandel.Pechenik} $q$ times.
\end{proof}

In fact, the orbit structure of $K$-promotion is entirely determined by its action on packed tableaux. For $T \in \inc^q(P)$, its \newword{binary content vector} is the $0$--$1$ vector of length $q$ whose $i$th coordinate is $1$ if the value $i$ appears as a label in $T$, and is $0$ if it does not. The following result is \cite[Theorem~6.1]{Mandel.Pechenik}.

\begin{proposition}[\cite{Mandel.Pechenik}]\label{prop:deflationperiod}
	Let $T \in \inc^q(P)$ with binary content vector $v$. Let $\ell$ be the least positive integer such that moving the first digit of $v$ to the end $\ell$ times recovers $v$. Then the $\psi$-orbit of $T$ has cardinality
	\[
	\tau = \frac{\ell \tau'}{\gcd(\ell d/q, \tau')},
	\]
	where $\tau'$ is the cardinality of the $\psi$-orbit of $T^\flat \in \inc^d(P)$.
\end{proposition}

\begin{lemma}\label{lem:packedRP}
	Let $P$ be a finite poset. Suppose $T \in \inc^q(P)$ is a tableau whose $\psi$-orbit has cardinality $k$ with $\gcd(k,q)=1$. Then $T$ is packed.
\end{lemma}
\begin{proof}
Let $v$ be the binary content vector of $T$ and let $\ell$ be the least positive integer such that moving the first digit of $v$ to the end $\ell$ times recovers $v$. Clearly, $\ell$ divides $q$.

Let $k'$ denote the cardinality of the $\psi$-orbit of $T^\flat \in \inc^d(P)$. Then, by Proposition~\ref{prop:deflationperiod}, we have
\[
k = \frac{\ell k'}{\gcd(\ell d/q, k')}.
\]
Since the denominator of this fraction is transparently a divisor of $k'$, this implies that $k$ is a multiple of $\ell$. Therefore, $\gcd(k,q) \geq \ell$.  But we are assuming that $\gcd(k,q)=1$, so we must have $\ell = 1$. It follows that $T$ is packed.
\end{proof}

The above results give rise to the following corollary, which will be very useful to us.

\begin{corollary}\label{cor:check_packed}
	Let $P$ be a finite graded poset. Then $P$ has NRP rowmotion if and only if, for every $c>0$, no packed increasing tableau $T \in \inc^{\rank(P) + c + 1}(P)$ has $K$-promotion orbit cardinality relatively prime to its alphabet size $\rank(P) + c + 1$.
\end{corollary}
\begin{proof}
	The equivariant bijection of Theorem~\ref{thm:equivbij} shows that $P$ having NRP rowmotion is equivalent to there existing no $T \in \inc^{\rank(P) + c + 1}(P)$, for any $c>0$, whose $K$-promotion orbit has cardinality relatively prime to $\rank(P) + c + 1$.

Lemma~\ref{lem:packedRP} shows that it is sufficient to check this condition for packed tableaux.
\end{proof}

Note that, for any particular finite graded poset $P$, Corollary~\ref{cor:check_packed} turns checking whether $P$ has NRP rowmotion into a finite verification.
In Section~\ref{sec:proof}, we will apply Corollary~\ref{cor:check_packed} to the minuscule posets to prove Theorem~\ref{thm:main}. 

Related to $K$-promotion is the involution \newword{$K$-evacuation} on increasing tableaux. $K$-evacuation was introduced in \cite{Thomas.Yong:K}, where helpful examples may be found.
For $T \in \inc^q(P)$ and $k \leq q$, let $T_{\leq k}$ denote the restriction of $T$ to the order ideal of $P$ labeled by elements less than or equal to $k$. Then the $K$-evacuation of $T$ is the unique increasing tableau $\mathcal{E}(T) \in \inc^q(P)$ such that, for every $k$, $\mathcal{E}(T)_{\leq k}$ and $\psi^{q-k}(T)_{\leq k}$ are defined on the same subposet of $P$. We will use $K$-evacuation in Section~\ref{sec:proof} as an ingredient in our proof of Theorem~\ref{thm:main}.

\section{Minuscule increasing tableaux}\label{sec:proof}
In this section, we study the $K$-promotion of increasing tableaux on minuscule posets. The main results of this section are a minuscule analogue of \cite[Theorem~2]{Pechenik:frames} and a proof of Theorem~\ref{thm:main}.

First, we prove a doubletree theorem, extending \cite[Theorem~2]{Pechenik:frames} from rectangles to arbitrary minuscule posets. For $M$ a minuscule poset and $T \in \inc(M)$, we write $\DT(T)$ for the restriction of $T$ to the subposet $\DT(M)$.

\begin{theorem}\label{thm:doubletree}
	Let $T$ be an increasing tableau of height $q$ on the minuscule poset $M$. Then 
	\[\DT(T) = \DT \big(\psi^{q}(T) \big).\]
\end{theorem}
\begin{proof}
	By Proposition~\ref{prop:deflation}, it is sufficient to prove the theorem for packed tableaux $T$, so assume $T$ is packed.
	
	We prove the theorem in cases, according to the classification of minuscule posets described in Section~\ref{sec:minuscule}. For rectangles, the theorem is exactly \cite[Theorem~2]{Pechenik:frames}. For propellers and the Cayley--Moufang poset, the stronger result $T = \psi^q(T)$ was shown in \cite[\textsection 7]{Mandel.Pechenik}. For the Freudenthal poset, the theorem may be verified explicitly for the $624493$ packed tableaux of that shape; we checked this fact using SageMath \cite{sagemath}.
	
	It remains to consider the case of shifted staircases. Given an increasing tableau $T$ of shape $S_k$, its \newword{doubling} is the increasing tableau $T^2$ of rectangular shape ${\bf k} \times {\bf k}$ given by gluing together two copies of $T$ as follows. Identifying ${\bf k} \times {\bf k}$ with the elements of $[k] \times [k]$ and identifying $S_k$ with the subset of pairs $(x,y) \in [k] \times [k]$ with $x \leq y$ (as described in Section~\ref{sec:minuscule}), we define $T^2(x,y) = T^2(y,x) = T(x,y)$ for all $1 \leq x \leq y \leq k$. Conversely, given a tableau $U$ of rectangular shape ${\bf k} \times {\bf k}$ that is symmetric in the sense that $U(x,y) = U(y,x)$ for all $x,y$, we define its \newword{radical} to be the increasing tableau $\sqrt{U}$ of shape $S_k$ given by restricting to the set of ordered pairs $(x,y)$ with $x \leq y$.
	
	Now, let $T$ be a packed increasing tableau of shape $S_k$ and consider its doubling $T^2$. By \cite[Proposition~17]{Pechenik:frames}, we have that $\PD(T^2)$ and $\mathcal{E}(T^2)$ have the same doubletree. Now it is easy to see that $\sqrt{\PD(T^2)} = \PD(T)$. The analogous statement $\sqrt{\mathcal{E}(T^2)} = \mathcal{E}(T)$ is immediate from \cite[Equation~(7.1)]{Buch.Samuel}. It follows that $\PD(T)$ and $\mathcal{E}(T)$ agree on the restriction of the doubletree of ${\bf k} \times {\bf k}$ to $S_k$, namely the nodes $(x,y)$ with $x=1$ or $y=k$. These nodes make up most of the doubletree of $S_k$.
	
	Next, we study the nodes $(2,2)$ and $(k-1,k-1)$, the remaining nodes of $\DT(S_k)$. An \newword{h-strip} in an increasing tableau $U$ of shape $S_k$ is a sequence of nodes $(x_1, y_1), \dots, (x_t,y_t)$ such that, for all $1\leq i < t$, we have $y_i < y_{i+1}$, $x_i \geq x_{i+1}$, and $U(x_i,y_i) \leq U(x_{i+1},y_{i+1})$. Similarly, a \newword{v-strip} is a sequence of nodes $(x_1, y_1), \dots, (x_t,y_t)$ such that we have $x_i < x_{i+1}$, $y_i \geq y_{i+1}$, and $U(x_i,y_i) \leq U(x_{i+1},y_{i+1})$.
	A \newword{Pieri strip} in an increasing tableau $U$ of shape $S_k$ is a sequence of distinct nodes $(x_1, y_1), \dots, (x_t,y_t), \dots (x_u,y_u)$ such that $(x_1, y_1), \dots, (x_t,y_t)$ forms a v-strip, $(x_{t+1}, y_{t+1}), \dots, (x_u,y_u)$ forms an h-strip, and we have $U(x_i,y_i) \leq U(x_{i+1},y_{i+1})$ for all $1 \leq i < u$. (Pieri strips are so named because they are related to Pieri classes in the cohomology of an orthogonal Grassmannian.)
	
	 By \cite[Theorem~4.6]{Clifford.Thomas.Yong}, $\mathcal{E}(T)_{\leq a}$ will be entirely supported on nodes of the form $(1,y)$ if and only if $T_{>n-a}$ forms a Pieri strip. Now observe that if $T_{>n-a}$ contains the node $(k-1,k-1)$, then it also contains the nodes $(k-1,k)$ and $(k,k)$, but these three nodes can never appear together in any Pieri strip of any tableau. Conversely, if $T_{>n-a}$ does not contain the node $(k-1,k-1)$, then it contains only nodes of the form $(i,k)$ and forms a Pieri strip (just a v-strip, in fact). Hence, $\mathcal{E}(T)_{\leq a}$ will be entirely supported on nodes of the form $(1,y)$ if and only if the node $(k-1,k-1)$ does not appear in $T_{>n-a}$. It follows that $\mathcal{E}(T)(2,2) = \PD(T)(2,2)$. Similarly, we may argue that $\mathcal{E}(T)(k-1,k-1) = \PD(T)(k-1,k-1)$. Thus, we have $\DT(\PD(T)) = \DT(\mathcal{E}(T))$.
	 
	  Since $\DT(\PD(T)) = \DT(\mathcal{E}(T))$ for all tableaux $T \in \inc^q(S_k)$, we also have
	  \[
	  \DT(T) = \DT(\PD(\mathcal{E}(\PD(\mathcal{E}(T))))).
	  \] However, we also have the equivalence of operators 
	  \[
	  \psi^q = \PD \circ \mathcal{E} \circ \PD \circ \mathcal{E}
	  \] by inspection of $K$-theoretic growth diagrams, as explained for rectangles in \cite[Lemma~3.1]{Pechenik:CSP}. Thus, we conclude 
	  \[\DT(T) = \DT \big(\psi^{q}(T) \big),\]
	  as desired.
	 	\end{proof}

In the following sections of this paper, we will only have need of the shifted staircase cases of Theorem~\ref{thm:doubletree}. 
We note that while many results for minuscule posets extend to more general $d$-complete posets (e.g., \cite{Ilango.Pechenik.Zlatin,Kim.Yoo,Naruse.Okada,Proctor.Scoppetta}), small calculations indicate that Theorem~\ref{thm:doubletree} does not extend in this fashion, although \cite[Conjecture~44]{Ilango.Pechenik.Zlatin} seems relevant.

\subsection{Proof of Theorem~\ref{thm:main}}

Finally, we turn to proving the main theorem of this paper. By Theorem~\ref{thm:equivbij}, it is equivalent to establish the following result on $K$-promotion of increasing tableaux of minuscule shape.

\begin{proposition}\label{prop:maintab}
	Let $M$ be a minuscule poset and let $q > \rank(M)+1$. Suppose the $\psi$-orbit of $T \in \inc^q(M)$ has cardinality $k$. Then $\gcd(k,q)>1$.
\end{proposition}

For $q = \rank(M)+1$, there is a unique increasing tableau $T \in \inc^q(M)$. Following \cite{Buch.Samuel}, we call this unique tableau the \newword{minimal tableau} $T_M$. Clearly, the $\psi$-orbit of $T_M$ has cardinality $1$. Under the equivariant bijection of Theorem~\ref{thm:equivbij}, $T_M$ corresponds to the unique order ideal of the empty poset $M \times {\bf 0}$.

We will need the following lemma, a shifted staircase analogue of \cite[Proposition~3.2]{Patrias.Pechenik}. For $V \in \inc^z(P)$, the \newword{flow path} of $V$ is the set of pairs $\{p \lessdot p'\}$ of poset elements such that $p'$ covers $p$ and both appear in the same nontrivial tile during some stage of the application of $\psi$ to the tableau $V$. The union of all the pairs in the flow path of $V$ is its \newword{streambed}. Note that if $p$ is in the streambed of $V$ and is neither maximal nor minimal in $P$, then it appears as the larger element of at least one pair in the flow path and also appears as the smaller element of at least one pair in the flow path.

\begin{lemma}\label{lem:minimalstaircase}
	Let $P$ be a graded poset with a unique maximum element $\hat{1}$ and a unique minimum element $\hat{0}$ and such that $\BT(P) \cap \TT(P) \neq \emptyset$. Suppose that $V \in  \inc^z(P)$ is any tableau satisfying $\DT(V) = \DT(\psi(V))$. Then $V = T_{P}$ is the minimal tableau and $z = \rank(P)+1$.
\end{lemma}
\begin{proof}
We must have $V(\hat{0}) = 1$, for otherwise $\psi(V)(\hat{0}) = V(\hat{0}) -1$, contradicting $\DT(V) = \DT(\psi(V))$. For the same reason, every $d \in \DT(V)$ must be in the streambed of $V$.

For $b \in \BT(P)$ with $b \neq \hat{0}$, there is a unique $b' \in P$ that $b$ covers. Since $b$ is in the streambed of $V$, $\{b' \lessdot b\}$ must then be a pair in the flow path. Therefore, $\psi(V)(b') = V(b) - 1$. But by assumption, $\psi(V)(b') = V(b')$. It follows that $V$ labels the bottom tree of $P$ as in the minimal tableau $T_P$.

Since $\hat{1}$ is in the streambed of $V$, we have $\psi(V)(\hat{1}) = z$. Therefore, since $\psi(V)(\hat{1}) = V(\hat{1})$ by assumption, we see that $V(\hat{1}) = z$. Now, for each $t \in \TT(P)$ with $t \neq \hat{1}$, there is a unique $t'$ that covers $t$. Since $t$ is in the streambed of $V$, the pair $\{t \lessdot t'\}$ must be in the flow path, so $\psi(V)(t) = V(t') - 1$. But by assumption $\psi(V)(t) = V(t)$, so $V(t) = V(t') -1$. 

Consider any maximal chain from $\hat{0}$ to $\hat{1}$ that is contained in $\DT(P)$. Such chains exist since we assume $\BT(P) \cap \TT(P) \neq \emptyset$. The tableau $V$ labels elements of that chain with consecutive positive integers, starting at $1$ and ending at $z$. But a maximal chain of $P$ has length $\rank(P)$. So $z = \rank(P)+1$. Since $T_P$ is the unique element of $\inc^{\rank(P)+1}(P)$, we must have $V = T_P$.
\end{proof}

\begin{proof}[Proof of Proposition~\ref{prop:maintab}]
Let $M$ be a minuscule poset and let $q \geq \rank(M)+1$. Suppose $T \in \inc^q(M)$ is a tableau whose $\psi$-orbit has cardinality $k$ with $\gcd(k,q)=1$. We will show that $T$ is the minimal tableau $T_M$.

By Lemma~\ref{lem:packedRP}, $T$ is a packed tableau.

We now break into cases according to the classification of minuscule posets from Section~\ref{sec:minuscule}. The rectangle case of the proposition is \cite[Theorem~2.4]{Patrias.Pechenik}. For the propellers and the Cayley-Moufang poset, the explicit calculations of \cite{Mandel.Pechenik} show that every packed tableau $T \in \inc^q(M)$ has $k$ dividing $q$, with $k=1$ only for the minimal tableau $T_M$. Similarly, for the Freudenthal poset, the computations of \cite{Mandel.Pechenik} show that every packed tableau $T \in \inc^q(M)$ either has $k$ dividing $q$, $k=2q$, or $k=3q$; moreover, $k=1$ only for the minimal tableau $T_M$.

Thus, it only remains to handle the shifted staircases. Suppose $T \in \inc^q(S_m)$. 
Consider the cyclic group $C_k = \langle g \rangle$ of order $k$. Let $C_k$ act on the $\psi$-orbit of $T$ by $g \cdot U = \psi(U)$ for all tableaux $U$ in the $\psi$-orbit. Since $\gcd(k,q) = 1$, the element $g^q$ is a generator of $C_k$. Thus, every $U$ in the $\psi$-orbit of $T$ is of the form $\psi^{qh}(T)$ for some positive integer $h$. 

Theorem~\ref{thm:doubletree} tells us that $\DT(V) = \DT(\psi^q(V))$ for all $V \in \inc^q(S_m)$. Thus, it follows from the previous paragraph that every $U$ in the $\psi$-orbit of $T$ satisfies $\DT(U) = \DT(T)$. Therefore, $\DT(T) = \DT(\psi(T))$.
Since $\BT(S_m) \cap \TT(S_m) \neq \emptyset$, we have then by Lemma~\ref{lem:minimalstaircase} that $T = T_{S_m}$, as desired. 
\end{proof}

Theorem~\ref{thm:main} is now immediate by combining Proposition~\ref{prop:maintab} with Theorem~\ref{thm:equivbij}.
\qed

\section{Other posets}
\label{sec:other}

In Theorem~\ref{thm:main}, we showed that all minuscule posets have NRP rowmotion. In this section, we consider possible extensions of Theorem~\ref{thm:main} to other posets. 

We begin with our positive results in this direction. The \newword{comparability graph} of a poset $P$ is the graph $G(P)$ whose vertex set is the set of elements of $P$ and with vertices $x,y$ adjacent in $G(P)$ if and only if $x$ and $y$ are comparable in $P$. Many properties of posets that depend only on the graph-isomorphism class of the comparability are collected and discussed in \cite{Hopkins:doppelganger}. The following theorem gives yet another such property.

\begin{theorem}\label{thm:iso_graphs}
	Let $P$ and $Q$ be finite posets with isomorphic comparability graphs, and let $q \in \mathbb{Z}^+$ be a positive integer. Then there is a cardinality-preserving bijection between the $K$-promotion orbits of $\inc^q(P)$ and the $K$-promotion orbits of $\inc^q(Q)$.
\end{theorem}
\begin{proof}
		We will need the characterization of having isomorphic comparability graphs given in \cite{Dreesen.Poguntke.Winkler} (although known earlier to various people). We will borrow the relevant terminology from \cite{Hopkins:doppelganger}. 
		
		An \newword{autonomous} subset of $P$ is a set $A \subseteq P$ such that each $b \in P \setminus A$ relates to every element of $A$ in the same way; that is to say, for all  $b \in P \setminus A$ and $a,a' \in A$, we have
		\begin{align*}
			a < b \; &\text{if and only if} \; a' < b, \; \text{and} \\
			a > b\; &\text{if and only if} \; a' > b.
		\end{align*} 
		The poset $P'$ \newword{obtained by dualizing $A$} is the poset obtained from $P$ by reversing the relation of each pair of comparable elements of $A$ while leaving all other relations the same; that is, for $p_1, p_2 \in P$,
		\begin{align*}
			p_1 <_{P'} p_2 \; &\text{if and only if} \; p_1 <_P p_2, \; &\text{if $p_1, p_2$ not both in $A$;} \\
			p_1 >_{P'} p_2 \; &\text{if and only if} \; p_1 <_P p_2, \; &\text{if $p_1, p_2 \in A$.}
		\end{align*} 
		
		It is clear that dualizing an autonomous subset preserves the comparability graph. The theorem we need from \cite{Dreesen.Poguntke.Winkler} is the converse statement that any finite posets with isomorphic comparability graphs are related by a finite sequence of dualizations of autonomous subsets.
		
		By induction, we may therefore assume $P$ and $Q$ are related by one such dualization. So let $A \subseteq P$ be an autonomous subset of $P$ such that $Q$ is the subset obtained by dualizing $A$. In particular, we will treat $P$ and $Q$ as living on the same ground set.
		
		Let $T \in \inc^q(P)$ be any increasing tableau. We define a tableau $f(T) \in \inc^q(Q)$ by flipping the labels on $A$ as follows. 
		Let $\{v_1 <v_2 < \dots < v_k \}$ be the set of labels appearing on the subposet $A$ in $T$.
		For $b \in Q \setminus A$, define $f(T)(b) = T(b)$. For $a \in A$ with $T(a) = v_i$, define $f(T)(a) = v_{k+1-i}$. Note that $f(T)$ is an increasing tableau since $A$ is autonomous.
		
		Let $R$ be the poset obtained from $P$ by deleting $A$ and replacing it with a $k$ element chain $C = \{c_1 < \dots < c_k\}$ with the same relations to $P \setminus A$; that is, for $a \in A$, $c \in C$, and $b \in P \setminus A$, we have
		\begin{align*}
			b <_R c \; &\text{if and only if} \; b <_P a, \; \text{and} \\
			b >_R c\; &\text{if and only if} \; b >_P a.
		\end{align*} 
		Note that $R$ may equally well be defined by replacing $A$ with $C$ in $Q$.
		
		We define $g(T) \in \inc^q(R)$ as follows. Define $g(T)(c_i) = v_i$, while for $b \in R \setminus C$, define $g(T)(b) = T(b) = f(T)(b)$.

		Recall the notions of flow paths and streambeds from Section~\ref{sec:proof}. We consider the actions of $K$-promotion on $T$ and $f(T)$. Clearly, if the label $1$ does not appear in $T$, then $f(\psi(T)) = \psi(f(T))$ and $g(\psi(T)) = \psi(g(T))$.
		Otherwise, note that the streambed of $T$ intersects $A$ if and only if the streambed of $f(T)$ intersects $A$ if and only if the streambed of $g(T)$ contains $C$. Moreover, the three flow paths all coincide on $P \setminus A$.
		
		Suppose the streambed of $T$ intersects $A$. Then the restriction of the flow path of $T$ to $A$ coincides with the flow path of the tableau $h(T)$ obtained by restricting $T$ to $A$ and then deflating it.
		Similarly, the restriction of the flow path of $f(T)$ to $A$ coincides with the flow path of the tableau obtained by restricting $f(T)$ to $A$ and then deflating it. Since these restricted deflated tableaux are exactly duals of each other, it follows that the restriction of the flow path of $f(T)$ to $A$ is the flow path for the inverse of $K$-promotion as applied to $h(T)$.
		
		For any $i$, we can now describe the effect of $\psi^i$ on $T$ and $f(T)$ in terms of $K$-promotion on $g(T)$ and $h(T)$. For any $b \in P \setminus A$, we have $\psi^i(T)(b) = \psi^i(f(T))(b) =  \psi^i(g(T))(b)$. Considering the $i$ streambeds from  applying $\psi$ to $g(T)$ $i$ times, let $j \leq i$ be the number of these streambeds that contain $C$. Suppose the labels of $C$ in $\psi^i(h(T))$ are $\{w_1 < w_2 < \dots < w_k \}$. Then using Proposition~\ref{prop:deflation}, for any $a \in A$, we have $\psi^i(T)(a) = w_\ell$ if and only if $\psi^j(h(T)) = \ell$ and we also have $\psi^i(f(T))(a) = w_{k+1-m}$ if and only if $\psi^{-j}(h(T)) = m$. In particular, the $\psi$-orbit of $T$ has the same cardinality as the $\psi$-orbit of $f(T)$. Since $T$ was arbitrary, $f$ then induces the desired cardinality-preserving bijection between the $K$-promotion orbits of $\inc^q(P)$ and $\inc^q(Q)$.
\end{proof}

For posets $P$ and $Q$, let $P \oplus Q$ denote the ordinal sum of $P$ and $Q$, where all elements of $P$ are declared to be less than all elements of $Q$. 
The following corollary is closely related to some speculations by Sam Hopkins and Mike Joseph in email correspondence (October 2020) with the author and others.

\begin{corollary}\label{cor:weird_propellers}
	If $P$ and $Q$ are finite graded posets with isomorphic comparability graphs and $c \in \mathbb{Z}^+$ is a positive integer, then there is a cardinality-preserving bijection between the rowmotion orbits of $P \times \bfc$ and those of $Q \times \bfc$, so that $P$ has NRP rowmotion if and only if $Q$ does.
	
	In particular, if $P$ is such that the ordinal sum $P \oplus {\bf 1}$ has NRP rowmotion, then so does ${\bf 1} \oplus P$.
\end{corollary}
\begin{proof}
The first statement is immediate from combining Theorems~\ref{thm:equivbij} and~\ref{thm:iso_graphs}. The second follows since $P \oplus {\bf 1}$ and ${\bf 1} \oplus P$ are related by a sequence of two dualizations of autonomous subsets.
\end{proof}

Corollary~\ref{cor:weird_propellers} (especially the ``in particular'' statement) gives some new examples of posets with NRP rowmotion by dualizing autonomous subsets of minuscule posets.

Let $P_{a,b}$ denote the poset ${\bf a} \oplus ({\bf 2} \times {\bf 2}) \oplus \bfb$. Note that $P_{a,a}$ is the propeller $J^a({\bf 2} \times {\bf 2})$, but otherwise $P_{a,b}$ is not a minuscule poset. Since propellers have NRP rowmotion, it follows from Corollary~\ref{cor:weird_propellers} that all posets $P_{a,b}$ with $a+b$ even have NRP rowmotion. Similarly, we may apply Corollary~\ref{cor:weird_propellers} to the other minuscule posets to obtain some additional posets with NRP rowmotion.

\begin{figure}[ht]
\begin{tikzpicture}
	\node (0) at (0,0) {$\bullet$};
\node [above of=0] (2) {$\bullet$};
\node [left of=2] (1) {$\bullet$};
\node [right of=2] (3) {$\bullet$};
\node [above of=2] (4) {$\bullet$};
\node [above of=4] (5) {$\bullet$};
\node [above of=5] (6) {$\bullet$};
\node [above of=6] (7) {$\bullet$};
\node [above of=7] (8) {$\bullet$};
\draw (0) to (1);
\draw (0) to (2);
\draw (0) to (3);
\draw (1) to (4);
\draw (2) to (4);
\draw (3) to (4);
\draw (4) to (5);
\draw (5) to (6);
\draw (6) to (7);
\draw (7) to (8);
\end{tikzpicture}
\caption{The poset $N$, one of the non-minuscule $9$-element posets with NRP rowmotion.}
\label{fig:posetN}
\end{figure}
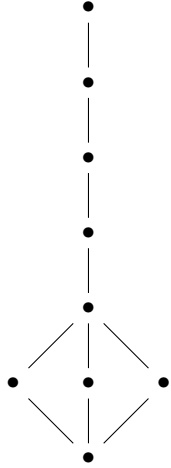

Using Corollary~\ref{cor:check_packed}, we systemically found all bounded graded posets on at most $9$ vertices that have NRP rowmotion. Every such poset with at most $8$ vertices either is minuscule or is one of the posets $P_{a.b}$ derived from a propeller via iterating the ``in particular'' statement of Corollary~\ref{cor:weird_propellers}. The $9$-vertex posets with NRP rowmotion are the minuscule posets ${\bf 3} \times {\bf 3}$ and ${\bf 9}$; the poset $N$ shown in Figure~\ref{fig:posetN} and the four other posets obtained from it by Corollary~\ref{cor:weird_propellers}; and the poset $W$ shown in Figure~\ref{fig:posetW}, together with its dual.

We note that the poset $W$ is related to the positive root poset $\Phi^+(B_3)$ via one application of the ``in particular'' statement of Corollary~\ref{cor:weird_propellers}. Rowmotion and $K$-promotion on these root posets has been studied in \cite{Dao.Wellman.Yost-Wolff.Zhang}.
Further discussion of root posets in relation to NRP rowmotion will appear in subsequent work with Marina Simmons.

The NRP rowmotion of the poset $N$ can be explained as follows. Observe that $N$ may be written as $N' \oplus {\bf 4}$, where $N'$ is a $5$-element graded poset. Consider the packed increasing tableaux on $N'$, which necessarily have alphabet size $3 = \rank(N')$, $4$, or $5=|N'|$. In $\inc^3(N')$, there is $1$ such tableau, forming a $1$-cycle under $\psi$. In $\inc^4(N')$, there are $6$ such tableaux, forming three $2$-cycles under $\psi$. In $\inc^5(N')$, there are again $6$ such tableaux, this time forming two $3$-cycles under $\psi$. Now consider the ordinal sum $N' \oplus \bfc$ for various $c \geq 0$. Clearly, the packed increasing tableaux of $N' \oplus \bfc$ have the same orbit structures as those of $N'$, but with the alphabet sizes increased by $c$. It follows from Corollary~\ref{cor:check_packed} that $N' \oplus \bfc$ will have NRP rowmotion whenever $c$ satisfies the system of congruences
\begin{align*}
	3+c &\equiv 0 \pmod 1, \\
	4+c &\equiv 0 \pmod 2,\\
	5+c &\equiv 0 \pmod 3.
\end{align*} By the Chinese Remainder Theorem, we may take $c=4$ to obtain the poset $N$. 

This construction of $N$ from $N'$ does not, however, extend to a general recipe for constructing posets with NRP rowmotion from arbitrary graded posets by taking the ordinal sum with a chain. Unfortunately, the analogous system of congruences one would have to solve generally has no solution.

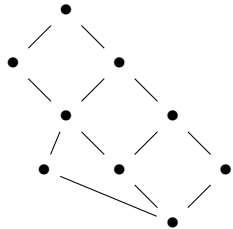
\begin{figure}[ht]
\begin{tikzpicture}
	\node (0) at (0,0) {$\bullet$};
\node [above left of=0] (1) {$\bullet$};
\node [above right of=0] (4) {$\bullet$};
\node [above left of=1] (3) {$\bullet$};
\node [left of=1] (2) {$\bullet$};
\node [above right of=1] (5) {$\bullet$};
\node [above right of=3] (6) {$\bullet$};
\node [above left of=3] (7) {$\bullet$};
\node [above right of=7] (8) {$\bullet$};
\draw (0) to (1);
\draw (0) to (2);
\draw (0) to (4);
\draw (1) to (3);
\draw (1) to (5);
\draw (2) to (3);
\draw (3) to (7);
\draw (3) to (6);
\draw (4) to (5);
\draw (5) to (6);
\draw (6) to (8);
\draw (7) to (8);
\end{tikzpicture}
\caption{The poset $W$, one of the non-minuscule $9$-element posets with NRP rowmotion. Note that relocating the minimum element of $W$ to its top would yield the root poset $\Phi^+(B_3)$.}
\label{fig:posetW}
\end{figure}

We now turn to some other families of posets, related to the minuscule posets, where one might naturally hope to find extensions of Theorem~\ref{thm:main}. We find, however, that these related families do not generally exhibit NRP rowmotion.

 Minuscule posets are examples of $d$-complete posets, posets associated to special dominant weights of Kac--Moody groups. See \cite{Proctor:Dynkin,Proctor:dcomplete,Proctor.Scoppetta} for background on $d$-complete posets.
Many results for minuscule posets extend to more general $d$-complete posets (e.g., \cite{Ilango.Pechenik.Zlatin,Kim.Yoo,Naruse.Okada,Proctor.Scoppetta}).
However, Theorem~\ref{thm:main} and Proposition~\ref{prop:maintab} do not extend to general $d$-complete posets. We call the poset $H$ whose Hasse diagram is illustrated in Figure~\ref{fig:Kpromotion} the \newword{bee hummingbird}. The bee hummingbird is a $d$-complete poset. In R.~Proctor's classification \cite{Proctor:Dynkin} of irreducible $d$-complete posets, the bee hummingbird is the smallest  exemplar of the family of $d$-complete posets known as \emph{birds}. It is easy to check that $\inc^6(H)$ contains a $\psi$-orbit consisting of exactly $5$ packed tableaux. Since $\gcd(5,6)=1$, it follows from Theorem~\ref{thm:equivbij} that the bee hummingbird does not have NRP rowmotion. Experimentation suggests that this is a general phenomenon for $d$-complete posets that are not minuscule. Indeed, we have not discovered any $d$-complete posets with NRP rowmotion besides minuscule posets.


What is different about general $d$-complete posets that explains why they fail to have NRP rowmotion? The issue may be that $d$-complete posets generally don't have unique maximum elements and don't have an analogue of the $\PD$ anti-involution (Kac--Moody Weyl groups are generally infinite and don't have a longest element).

As another source of posets with NRP rowmotion, one might hope that $P \times \bfc$ would have NRP rowmotion whenever $P$ does. However, this is entirely false. Indeed, although all rectangles have NRP rowmotion, the triple product of chains ${\bf 2} \times {\bf 2} \times {\bf 2}$ does not. By Theorem~\ref{thm:equivbij}, this failure is immediate from the packed tableau $T \in \inc^7({\bf 2} \times {\bf 2} \times {\bf 2})$ illustrated in Figure~\ref{fig:cube}, whose $\psi$-orbit has cardinality $27$.
 
 \begin{figure}[h]
 \begin{tikzpicture}[line cap=round,line join=round,>=triangle 45,x=4.0cm,y=4.0cm]
\clip(0,0.1) rectangle (1.2,1.0);
\draw [line width=1.pt] (0.6,0.2)-- (0.4,0.4);
\draw [line width=1.pt] (0.6,0.2)-- (0.6,0.4);
\draw [line width=1.pt] (0.6,0.2)-- (0.8,0.4);
\draw [line width=1.pt] (0.4,0.6)-- (0.6,0.8);
\draw [line width=1.pt] (0.6,0.8)-- (0.6,0.6);
\draw [line width=1.pt] (0.6,0.8)-- (0.8,0.6);
\draw [line width=1.pt] (0.4,0.6)-- (0.6,0.4);
\draw [line width=1.pt] (0.6,0.4)-- (0.8,0.6);
\draw [line width=1.pt] (0.6,0.6)-- (0.4,0.4);
\draw [line width=1.pt] (0.6,0.6)-- (0.8,0.4);
\draw [line width=1.pt] (0.4,0.6)-- (0.4,0.4);
\draw [line width=1.pt] (0.8,0.6)-- (0.8,0.4);
\draw [fill=ffqqqq] (0.6,0.2) circle (2.5pt) node[anchor=west] {$1$};
\draw [fill=ffqqqq] (0.4,0.4) circle (2.5pt) node[anchor=east] {$3$};
\draw [fill=ffqqqq] (0.6,0.4) circle (2.5pt) node[anchor=west] {$4$};
\draw [fill=ffqqqq] (0.8,0.4) circle (2.5pt) node[anchor=west] {$2$};
\draw [fill=ffqqqq] (0.4,0.6) circle (2.5pt) node[anchor=east] {$6$};
\draw [fill=ffqqqq] (0.6,0.6) circle (2.5pt) node[anchor=west] {$5$};
\draw [fill=ffqqqq] (0.8,0.6) circle (2.5pt) node[anchor=west] {$5$};
\draw [fill=ffqqqq] (0.6,0.8) circle (2.5pt) node[anchor=west] {$7$};
\end{tikzpicture}
 \caption{An increasing tableau $T \in \inc^7({\bf 2} \times {\bf 2} \times {\bf 2})$ whose $\psi$-orbit has cardinality $27$. Note that $\gcd(7,27)=1$, so $T$ witnesses that ${\bf 2} \times {\bf 2} \times {\bf 2}$ does not have NRP rowmotion.}	
 \label{fig:cube}
 \end{figure}
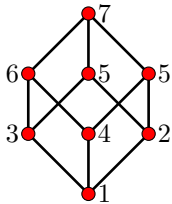
 
 Similar phenomena appear for other triple products of chain posets.  

\section*{Acknowledgements}
The author acknowledges support from NSERC Discovery Grant RGPIN-2021-02391 and Launch Supplement DGECR-2021-00010.

The author is grateful for useful conversations with Karen Collins, Sam Hopkins, Mike Joseph, Anne Schilling, Marina Simmons, Jessica Striker, and Corey Vorland. The author is also grateful to Sylvie Corteel for teaching him to use the word ``packed.'' Thanks also to Mike Joseph, Becky Patrias, and Jessica Striker for careful reading and very helpful comments on an earlier draft of this paper.

\bibliographystyle{amsalpha} 
\bibliography{CFDF}

\end{document}